\documentclass[reqno,12pt]{article}
\usepackage[left=1.2in, right = 1.2in, top=1in]{geometry}\usepackage{amsmath, amssymb, amsthm, mathrsfs, color, times, textcomp, verbatim}
\usepackage{amstext}
\usepackage{appendix}
\usepackage{verbatim}
\usepackage{cases}
\usepackage{pgfplots}
\usepackage{graphicx}
\usepackage{tikz}
\newtheorem{theorem}{Theorem}[section]

\newtheorem{corollary}{Corollary}[section]

\newtheorem{lemma}{Lemma}[section]

\allowdisplaybreaks[4]
\usepackage{CJKutf8}
\usepackage{fancyhdr}
\newcommand\keywords[1]{\textbf{Keywords}:#1}

\title{Liouville theorem for one kind of elliptic equations on complete Riemannian manifold}

\author{ Wangzhe Wu\\Institute of Mathematics \\Academy of Mathematics and Systems Science\\ Chinese
Academy of Sciences, Beijing, 100190, China\\Email: wuwz18@mail.ustc.edu.cn}

\begin{document}
\begin{CJK}{UTF8}{gbsn}

	\pagestyle{fancy}

\fancyhead{}

\fancyhead[CE]{\leftmark}

	\maketitle

	\begin{abstract}
		We use maximum principle to  prove the Liouville theorem of the equation $\Delta U + b\cdot \nabla U + h U^{\alpha} = 0, U \geq 0, 0 < \alpha < \frac{n + 2}{n - 2}$ on the complete Riemannian manifold with non-negative Ricci tensor, which improve the result of Gidas-Spruck and Catino-Monticelli.
	\end{abstract}
	
	
	\section{Introduction}
		A. Huber in \cite{Huber} and Cheng-Yau in \cite{Yau} proved the non-existence of non-constant positive harmonic function on the complete Riemannian manifold with non-negative Ricci tensor.
		B. Gidas and J. Spruck \cite{Gidas} used integration by parts to prove the Liouville theorem of the equation
		\begin{equation}\label{equ111}
			\Delta u + u^\alpha = 0,1 \leq \alpha < \frac{n + 2}{n - 2},  u \geq 0 \text{ in } M,
		\end{equation}
		where $M$ is a complete Riemannian manifold with non-negative Ricci tensor and $n > 2$. Later W. X. Chen and C. M. Li in \cite{Li} gave a new proof in $\mathbb R^n$ by using moving planes method. Besides B. Gidas and J. Spruck studied the properties of positive solutions of the equation
	\begin{equation}\label{equ1}
		\Delta U + b\cdot \nabla U + hU^{\alpha} = 0, 1 \leq \alpha < \frac{n + 2}{n - 2}, U\geq 0 \text{ in } M.
	\end{equation}	
	And they got the Liouville theorem under some assumptions. J. Y. Li in \cite{LJY}(his Theorem 2.2) used maximum principle to improve the result of B. Gidas and J. Spruck: he relaxed their conditions on the growth of $|\nabla \log (h)|, |b|$ and on the range of $\alpha$. But he only solved the case that
	$0 < \alpha < \frac{n + 4}{n}$ when $n = 2, 3$ and $0 < \alpha < \frac{n }{n - 2}$ when $n \geq 4$. Recently when $b \equiv 0$, Catino and Monticelli  in \cite{JEMS} got the Liouville theorem for $1 < \alpha < \frac{n + 2}{n - 2}$ where an extra condition 
	\begin{equation}\label{1204extra}
		h(x) \geq Cr(x)^{-\sigma}
	\end{equation}
	 is needed. In this paper, we use maximum principle to obtain the Liouville theorem for $0 < \alpha < \frac{n + 2}{n - 2}$ and get rid of the extra condition \eqref{1204extra}. Our idea originally comes from the gradient estimate of Cheng and Yau \cite{Yau}, where they used the auxiliary function 
	\begin{equation}
		Z = u^{-2}|\nabla u|^2.
	\end{equation}
	The group of Youde Wang \cite{WangYoude1, WangYoude2, WangYoude3, WangYoude4} used the auxiliary function and Moser iteration to get gradient estimate and Harnack inequality for semilinear and quasilinear equations.
	Besides the auxiliary function in \cite{LJY} and \cite{Brendle}  is developed to 
	\begin{equation}\label{1204Z1}
		Z = u^{\gamma -2}|\nabla u|^2 + \beta u^{\gamma -1}\Delta u,
	\end{equation}
	which will also be used in our proof.
	 And Zhihao Lu \cite{Lu} developed the method of J.Y. Li \cite{LJY} and choosed the auxiliary function
	\begin{equation}\label{1204Z2}
		Z = (u + \varepsilon)^{\gamma} (u^{ -2}|\nabla u|^2 + \beta u^{ -1}\Delta u),
	\end{equation}
	where he used maximum principle to get gradient estimate and Harnack inequality for the equation $\Delta u = f(u, x)$ for some special $f$. However compared \eqref{1204Z2} with \eqref{1204Z1}, his method is too complicated and our proof in this paper is simpler than his.	 Here is our main result:
	\begin{theorem}\label{theorem2}
		Let $M$ be a n-dimensional complete Riemannian manifold with Ricci tensor $R_{i j}$. Suppose that $h \in C^2(M)$ , $b \in \mathscr{X}(M)$, $\alpha \in \mathbb R$ and $R_{i j}$ satisfy the following conditions:
		\begin{itemize}
			\item (1) $\forall x \in M$, $h(x) \geq 0$, $\Delta h(x) + \nabla h(x) \cdot b(x) \geq 0$.
			\item (2) $0 < \alpha < \frac{n + 2}{n - 2}$.
			\item (3) $|b| = o(\rho^{-1})$ as $\rho \rightarrow \infty$ and when $\rho > 1$,
			\begin{equation}
				R_{i j} \geq -k(\rho)^2 g_{i j}, 
			\end{equation} 
			where $\rho = \rho(x)$ is the geodesic distance from $x$ to a fixed point $x_0$ and $k(\rho) = o(\rho^{-1})$.
			\item (4) The tensor field $-b_{i j} + R_{i j} - \left( C_{n, \alpha} - \frac{1}{n}\right) |b|^2g_{i j}$ is positive definite, where the constant $C(n, \alpha) > 0$ depends only on $n, \alpha$.
		\end{itemize} 
		Then every non-negative solution of \eqref{equ1} is constant.
	\end{theorem}
	This result improves the result of B. Gidas and J. Spruck when $b \equiv 0$ by relaxing the conditions on $h$ and $R_{i j}$.
	
	\begin{corollary}
		If $b \equiv 0$, $h \geq 0, \Delta h \geq 0$, $R_{i j} \geq 0$ and $0 < \alpha < \frac{n + 2}{n - 2}$, then every non-negative solution of \eqref{equ1} is constant.
	\end{corollary}

	\textbf{Acknowledgement } The author was supported by National Natural Science Foundation of China (Grants 11721101 and 12141105) and National Key Research and Development Project (Grants SQ2020YFA070080). The author would like to thank Prof. Xi-Nan Ma  for valuable comments and suggestions.

	\section{Proof of Theorems}\label{Proof of Theorems}
	We suggest the reader only consider the case $b_i \equiv 0$ when reading the paper for the first time.
	Consider the equation \eqref{equ1} and let $U(x)$ be a non-negative solution of \eqref{equ1}. Then by strong maximum principle, we know $U(x) > 0$. We consider the following auxiliary function:
	\begin{equation}
		Z := hU^{\gamma + \alpha} + \beta U^{\gamma - 1}|\nabla U|^2,
	\end{equation}
	with
	\begin{align*}
		\gamma &= - \frac{2(n - 1)(\alpha + \varepsilon_1 + \varepsilon_2 t) }{(n + 2)(1 - \varepsilon_2)} + 1,\\
		\beta &= \frac{n(1 - \varepsilon_2)(\gamma + \alpha)}{2(1 - n\varepsilon_2)} .
	\end{align*}
	where $\varepsilon_2 > 0$ is small enough and $\varepsilon_1 , t$ will be determined later.
	From now on, let us follow the idea of J. Y. Li in \cite{LJY}:
	\begin{lemma}\label{lemma1}
		Let $M$ be a n-dimensional complete Riemannian manifold with Ricci curvature $R_{i j}$. Suppose that $h \in C^2(M)$ , $b \in \mathscr{X}(M)$ and $\alpha \in \mathbb R$ satisfy the following conditions:
		\begin{itemize}
			\item (1) $\forall x \in M$, $h(x) \geq 0$, $\Delta h(x) + \nabla h(x) \cdot b(x) \geq 0$,
			\item (2) $0 < \alpha < \frac{n + 2}{n - 2}$.
		\end{itemize} 
		Then we can choose $\gamma, \beta, t$ and $\varepsilon > 0$, such that
		\begin{equation}
			\begin{aligned}
				\frac{1}{2\beta}(U^{t - \gamma + 1}Z_i)_i &\geq - \frac{1}{2\beta} U^{t - \gamma + 1}Z_i b_i -  \left(C_{n, \alpha} - \frac{1}{n}\right) |b|^2 U^t |\nabla U|^2  - U^{t} (b_{i j} - R_{i j}) U_i U_j  \\
		&+ \varepsilon U^{t + 2\alpha} h^2  + \varepsilon U^{t - 2}|\nabla U|^4.
			\end{aligned}
		\end{equation}
	\end{lemma}
	
	\begin{lemma}\label{lemma2}
		Let $M$ be a n-dimensional complete Riemannian manifold with Ricci curvature $R_{i j}$. Suppose that $h \in C^2(M)$ , $b \in \mathscr{X}(M)$ and $\alpha \in \mathbb R$ satisfy the following conditions:
		\begin{itemize}
			\item (1) $\forall x \in M$, $h(x) \geq 0$, $\Delta h(x) + \nabla h(x) \cdot b(x) \geq 0$,
			\item (2) $0 < \alpha < \min \left\{\frac{1}{n - 2}\left(\frac{2}{\sqrt n + 1} + n\right), \frac{n + 2}{n - 1} \right\}$,
			\item (3) $\gamma = - 1$.
		\end{itemize} 
		Then we can choose $\beta, t$ and $\varepsilon > 0$, such that
		\begin{equation}
			\begin{aligned}
				\frac{1}{2\beta}(U^{t - \gamma + 1}Z_i)_i &\geq - \frac{1}{2\beta} U^{t - \gamma + 1}Z_i b_i -  \left(C_{n, \alpha} - \frac{1}{n}\right)|b|^2 U^t |\nabla U|^2  - U^{t} (b_{i j} - R_{i j})U_i U_j  \\
		&+ \varepsilon U^{t + 2\alpha} h^2  + \varepsilon U^{t - 2}|\nabla U|^4.
			\end{aligned}
		\end{equation}
	\end{lemma}
	
	\begin{proof}[Proof of Lemma \ref{lemma1}]
		Let $e_1, e_2, ... , e_n$ be a local orthonormal frame field. By adopting the notation of moving frames, subscripts in $i, j$ and $l$ will denote covariant differentiation in the $e_i, e_j$ and $e_l$ directions, where $1 \leq i, j, l \leq n$. Suppose $b = b_i e_i$, then
	\begin{align*}
		Z_i &= h_i U^{\gamma + \alpha}  + (\gamma + \alpha)h U^{\gamma + \alpha - 1}U_i + \beta(\gamma - 1) U^{\gamma - 2}|\nabla U|^2U_i + 2\beta  U^{\gamma - 1}U_{i j}U_j,
	\end{align*}
	and
	\begin{align*}
		 &\frac{1}{2\beta}(U^{t - \gamma + 1}Z_i)_i\\
		 =&\frac{1}{2\beta} \Big[ h_i U^{t + \alpha + 1}  + (\gamma + \alpha)h U^{t + \alpha }U_i + \beta(\gamma - 1) U^{t - 1}|\nabla U|^2U_i + 2\beta  U^{t}U_{i j}U_j \Big]_i \\
		 =&\frac{1}{2\beta} h_{ii}U^{t + \alpha + 1} + \frac{1}{2\beta}(t + \alpha + 1)U^{t + \alpha} U_i h_i\\
		 &+ \frac{\gamma + \alpha}{2\beta}U^{t + \alpha} U_i h_i + \frac{\gamma + \alpha}{2\beta}U^{t + \alpha} \Delta U h  + \frac{\gamma + \alpha} {2\beta}(t + \alpha)U^{t + \alpha - 1} |\nabla U|^2 h\\
		 &+ \frac{\gamma - 1}{2}(t - 1) U^{t - 2}|\nabla U|^4 + (\gamma - 1)U^{t - 1}U_{i j}U_i U_j + \frac{\gamma - 1}{2}U^{t - 1}|\nabla U|^2 \Delta U\\
		 &+ tU^{t - 1}U_{i j}U_jU_i + U^{t}U_{j i i}U_j + U^{t}U_{i j}^2.
	\end{align*}
	Substitute \eqref{equ1} and using the fact that
	\begin{equation}
		U_{j ii } = U_{i i j} + R_{i j}U_i,
	\end{equation}
	we get
	\begin{align*}
		 &\frac{1}{2\beta}(U^{t - \gamma + 1}Z_i)_i\\
		 =&\frac{1}{2\beta} h_{ii}U^{t + \alpha + 1} + \frac{1}{2\beta}(t + \alpha + 1)U^{t + \alpha} U_i h_i\\
		 &+ \frac{\gamma + \alpha}{2\beta}U^{t + \alpha} U_i h_i + \frac{\gamma + \alpha}{2\beta}U^{t} \Delta U (-\Delta U - b_i U_i)  + \frac{\gamma + \alpha} {2\beta}(t + \alpha)U^{t - 1} |\nabla U|^2 (-\Delta U - b_i U_i) \\
		 &+ \frac{\gamma - 1}{2}(t - 1) U^{t - 2}|\nabla U|^4 + (\gamma - 1)U^{t - 1}U_{i j}U_i U_j + \frac{\gamma - 1}{2}U^{t - 1}|\nabla U|^2 \Delta U\\
		 &+ tU^{t - 1}U_{i j}U_jU_i + U^{t}U_j(U_{i i j} + R_{i j}U_i) + U^{t}U_{i j}^2,
	\end{align*}
	\begin{align*}
		 \Rightarrow &\frac{1}{2\beta}(U^{t - \gamma + 1}Z_i)_i\\
		 =&\frac{1}{2\beta} h_{ii}U^{t + \alpha + 1} + \frac{1}{2\beta}(t + \alpha + 1)U^{t + \alpha} U_i h_i\\
		 &+ \frac{\gamma + \alpha}{2\beta}U^{t + \alpha} U_i h_i - \frac{\gamma + \alpha}{2\beta}U^{t} \Delta U   b_i U_i  - \frac{\gamma + \alpha} {2\beta}(t + \alpha)U^{t - 1} |\nabla U|^2 b_i U_i \\
		 &+ \frac{\gamma - 1}{2}(t - 1) U^{t - 2}|\nabla U|^4 + (t + \gamma - 1)U^{t - 1}E_{i j}U_i U_j \\
		 &+ \alpha U^{t - 1}|\nabla U|^2 b_i U_i - U^{t}U_j( b_{i j} U_i + b_i U_{i j} + h_j U^{\alpha} ) + U^{t}U_jR_{i j}U_i + U^{t}E_{i j}^2\\
		 &+ \left( \frac{1}{n} - \frac{\gamma + \alpha}{2\beta}\right)U^t (\Delta U)^2 + \left[ - \frac{\gamma + \alpha} {2\beta}(t + \alpha) + \frac{t + \gamma - 1}{n} + \frac{\gamma - 1}{2} + \alpha \right]U^{t - 1}|\nabla U|^2\Delta U,
	\end{align*}
	\begin{equation}\label{final_equ1}
		\begin{aligned}
			\Rightarrow &\frac{1}{2\beta}(U^{t - \gamma + 1}Z_i)_i\\
		  =&\frac{1}{2\beta} h_{ii}U^{t + \alpha + 1} + \left[ \frac{1}{2\beta}(t + \alpha + 1) + \frac{1 - n\varepsilon_2}{n(1 - \varepsilon_2)} \right] U^{t + \alpha} U_i h_i - \frac{1 - n\varepsilon_2}{n(1 - \varepsilon_2)} U^{t} \Delta U   b_i U_i \\
		  &+ \left[ \alpha - \frac{1 - n\varepsilon_2}{n(1 - \varepsilon_2)}(t + \alpha)\right] U^{t - 1} |\nabla U|^2 b_i U_i + \frac{\gamma - 1}{2}(t - 1) U^{t - 2}|\nabla U|^4 + (t + \gamma - 1)U^{t - 1}E_{i j}U_i U_j \\
		 & - U^{t}U_j( b_{i j} U_i + b_i U_{i j} + h_j U^{\alpha} ) + U^{t}U_jR_{i j}U_i + U^{t}E_{i j}^2\\
		 &+  \frac{n - 1}{n}\cdot \frac{\varepsilon_2}{1 - \varepsilon_2}U^t (\Delta U)^2 - \frac{(n - 1)\varepsilon_1}{n(1 - \varepsilon_2)}U^{t - 1}|\nabla U|^2\Delta U.
		\end{aligned}
	\end{equation}
	On the other hand, we have
	\begin{align*}
		Z_i b_i &= h_ib_i U^{\gamma + \alpha}  + (\gamma + \alpha)h U^{\gamma + \alpha - 1}U_i b_i + \beta(\gamma - 1) U^{\gamma - 2}|\nabla U|^2U_i b_i+ 2\beta  U^{\gamma - 1}U_{i j}U_jb_i,
	\end{align*}
	\begin{equation}
		\Rightarrow U^{t - \gamma + 1}Z_i b_i = h_ib_i U^{t + \alpha + 1}  + (\gamma + \alpha)h U^{t + \alpha}U_i b_i + \beta(\gamma - 1) U^{t - 1}|\nabla U|^2U_i b_i+ 2\beta  U^{t}U_{i j}U_jb_i.
	\end{equation}
	Substitute it into \eqref{final_equ1}:
	\begin{equation}
		\begin{aligned}
			 &\frac{1}{2\beta}(U^{t - \gamma + 1}Z_i)_i\\
		  =&\frac{1}{2\beta} (h_{ii} + h_i b_i) U^{t + \alpha + 1} + \left[ \frac{1}{2\beta}(t + \alpha + 1) + \frac{1 - n\varepsilon_2}{n(1 - \varepsilon_2)} \right] U^{t + \alpha} U_i h_i \\
		  &+ \frac{1}{2\beta} \Big[ (\gamma + \alpha)h U^{t + \alpha}U_i b_i + \beta(\gamma - 1) U^{t - 1}|\nabla U|^2U_i b_i+ 2\beta  U^{t}U_{i j}U_jb_i - U^{t - \gamma + 1}Z_i b_i\Big]  \\
		 & + \frac{1 - n\varepsilon_2}{n(1 - \varepsilon_2)}U^{t} (b_j U_j + hU^{\alpha})   b_i U_i + \left[ \alpha - \frac{1 - n\varepsilon_2}{n(1 - \varepsilon_2)}(t + \alpha)\right]U^{t - 1} |\nabla U|^2 b_i U_i \\
		 &+ \frac{\gamma - 1}{2}(t - 1) U^{t - 2}|\nabla U|^4 + (t + \gamma - 1)U^{t - 1}E_{i j}U_i U_j \\
		 & - U^{t}U_j( b_{i j} U_i + b_i U_{i j} + h_j U^{\alpha} ) + U^{t}U_jR_{i j}U_i + U^{t}E_{i j}^2\\
		 &+  \frac{n - 1}{n}\cdot \frac{\varepsilon_2}{1 - \varepsilon_2}U^t (\Delta U)^2 - \frac{(n - 1)\varepsilon_1}{n(1 - \varepsilon_2)}U^{t - 1}|\nabla U|^2\Delta U.\\
		\end{aligned}
	\end{equation}
	\begin{equation}\label{final_equ2}
		\begin{aligned}
		\Rightarrow	 &\frac{1}{2\beta}(U^{t - \gamma + 1}Z_i)_i\\
		 =&- \frac{1}{2\beta} U^{t - \gamma + 1}Z_i b_i + \frac{1}{2\beta} (h_{ii} + h_i b_i) U^{t + \alpha + 1} + \left[ \frac{1}{2\beta}(t + \alpha + 1) + \frac{1 - n\varepsilon_2}{n(1 - \varepsilon_2)} - 1\right] U^{t + \alpha} U_i h_i\\
		  &+  2\frac{1 - n\varepsilon_2}{n(1 - \varepsilon_2)} h U^{t + \alpha}U_i b_i + \Bigg[ \frac{1}{2}(\gamma - 1) + \alpha - \frac{1 - n\varepsilon_2}{n(1 - \varepsilon_2)}(t + \alpha) \Bigg] U^{t - 1}|\nabla U|^2U_i b_i  \\
		 & + \frac{1 - n\varepsilon_2}{n(1 - \varepsilon_2)}U^t b_j U_j    b_i U_i  + U^{t}(R_{i j} - b_{i j})U_i U_j  \\
		 &+ \frac{\gamma - 1}{2}(t - 1) U^{t - 2}|\nabla U|^4 + (t + \gamma - 1)U^{t - 1}E_{i j}U_i U_j + U^{t}E_{i j}^2\\
		  &+  \frac{n - 1}{n}\cdot \frac{\varepsilon_2}{1 - \varepsilon_2}U^t (\Delta U)^2 - \frac{(n - 1)\varepsilon_1}{n(1 - \varepsilon_2)}U^{t - 1}|\nabla U|^2\Delta U.\\
		\end{aligned}
	\end{equation}
	Substitute \eqref{equ1} into \eqref{final_equ2} :
	
	\begin{align}\label{final_equ3}
		 &\frac{1 - \varepsilon_2}{2\beta}(U^{t - \gamma + 1}Z_i)_i + \frac{1 - \varepsilon_2}{2\beta} U^{t - \gamma + 1}Z_i b_i \notag\\
		  =&  \frac{1 - \varepsilon_2}{2\beta} (h_{ii} + h_i b_i) U^{t + \alpha + 1} + (1 - \varepsilon_2) \left[ \frac{1}{2\beta}(t + \alpha + 1) + \frac{1 - n\varepsilon_2}{n(1 - \varepsilon_2)} - 1\right] U^{t + \alpha} U_i h_i \notag\\
		  &+  2\frac{1 - \varepsilon_2}{n} h U^{t + \alpha}U_i b_i + (1 - \varepsilon_2)\Bigg[ \frac{1}{2}(\gamma - 1) + \alpha - \frac{1 - n\varepsilon_2}{n(1 - \varepsilon_2)}(t + \alpha) + \frac{(n - 1)\varepsilon_1}{n}\Bigg] U^{t - 1}|\nabla U|^2U_i b_i  \notag\\
		 & + \frac{1 - \varepsilon_2}{n}U^t b_j U_j    b_i U_i + (1 - \varepsilon_2) U^{t}(R_{i j} - b_{i j})U_i U_j  \notag\\
		 &+(1 - \varepsilon_2) \frac{\gamma - 1}{2}(t - 1) U^{t - 2}|\nabla U|^4 + (1 - \varepsilon_2) (t + \gamma - 1)U^{t - 1}E_{i j}U_i U_j + U^{t}E_{i j}^2 \notag\\
		  &+  \frac{n - 1}{n} \varepsilon_2 h^2 U^{t + 2\alpha} + \frac{(n - 1)\varepsilon_1}{n}hU^{t + \alpha - 1}|\nabla U|^2.
	\end{align}
	By the H\"older's inequality, we have
	\begin{equation}
		 U^{t - 1}|\nabla U|^2U_i b_i \leq \varepsilon U^{t - 2}|\nabla U|^4 + \frac{C}{\varepsilon} U^t b_j U_j    b_i U_i,
	\end{equation}
	and
	\begin{equation}
		h U^{t + \alpha}U_i b_i \leq \varepsilon  h^2 U^{t + 2\alpha} + \frac{C}{\varepsilon} U^t b_j U_j    b_i U_i.
	\end{equation}
	Define
	\begin{equation}
		F := \frac{4n}{n - 1}\cdot \frac{\gamma - 1}{2}(t - 1) - (t + \gamma - 1)^2.
	\end{equation}
	We hope that
	\begin{numcases}{}
		\label{1201equ1} \frac{1}{2\beta}(t + \alpha + 1) + \frac{1 - n\varepsilon_2}{n(1 - \varepsilon_2)} - 1 = 0,\\
		 F \geq 0,\\
		0 < \varepsilon_1 < 1,\\
		0 < \varepsilon_2 < \frac{1}{n} ,\\
		\label{1201equ2} \beta \geq 0.
	\end{numcases}
	If so, then the lemma holds naturally. In this paper, we always let $\varepsilon_2 > 0$ small enough but $\varepsilon_1$ is different.
	\begin{itemize}
		\item Let $\varepsilon_1, \varepsilon_2 = 0$, then by \eqref{1201equ1}
	\begin{equation}
		\begin{aligned}
			t&= (n - 1)\gamma + (n - 2)\alpha - 1\\ 
			&= - \frac{2(n - 1)^2\alpha  }{n + 2}  + (n - 2)\alpha + n - 2\\
			&= \frac{-n^2 + 4n - 6}{n + 2}\alpha + n - 2,
		\end{aligned}
	\end{equation}
	and
	\begin{equation}
		\begin{aligned}
			F=& \frac{-n^4 + 8n^3 - 28n^2 + 40n - 16}{(n + 2)^2}\alpha^2 + \frac{2n^3 - 12n^2 + 28n - 16}{n + 2}\alpha - (n - 2)^2.
		\end{aligned}
	\end{equation}
	We find that if 
	\begin{equation}
		\frac{n + 2}{n - 2}\cdot\frac{(n - 2)^3}{(n - 2)^3 + 4n} < \alpha < \frac{n + 2}{n - 2},
	\end{equation}
	then	
	\begin{equation}
		F > 0.
	\end{equation}
	
	\item If $\varepsilon_2= 0, \varepsilon_1 > 0$, we have
	\begin{equation}
		\begin{aligned}
			t &= (n - 1)\gamma + (n - 2)\alpha - 1\\
			&= - (n - 1)\frac{2(n - 1)(\alpha + \varepsilon_1) }{n + 2} + (n - 2)\alpha + n - 2\\
			&= \frac{-n^2 + 4n - 6}{n + 2}\alpha - \frac{2(n - 1)^2}{n + 2}\varepsilon_1 + n - 2,
		\end{aligned}
	\end{equation}
	and
	\begin{equation}
		\begin{aligned}
			\Rightarrow F &= \frac{-n^4 + 8n^3 - 28n^2 + 40n - 16}{(n + 2)^2}(\alpha + \varepsilon_1) ^2 + \frac{2n^3 - 12n^2 + 28n - 16}{n + 2}(\alpha + \varepsilon_1)(1 - \varepsilon_1)\\
			& - (n - 2)^2(1 - \varepsilon_1)^2 + \frac{4n}{n + 2}(\alpha + \varepsilon_1)\varepsilon_1. \\
		\end{aligned}
	\end{equation}
	
	So we get that if
	\begin{equation}\label{range}
		(1 - \varepsilon_1)\frac{n + 2}{n - 2}\cdot\frac{(n - 2)^3}{(n - 2)^3 + 4n} \leq \alpha + \varepsilon_1 \leq (1 - \varepsilon_1)\frac{n + 2}{n - 2} ,
	\end{equation}
	then
	\begin{equation}
		F > 0.
	\end{equation}
	We find that the parameter $\varepsilon_1$ is important: when $\alpha$ tends to $\frac{n + 2}{n - 2}$, we can choose  $\varepsilon_1$ tends to $0$; when $\alpha$ tends to $0$, we can choose  $ \varepsilon_1$ tends to $1$. But in general, we have 
	\begin{equation}
		\alpha + \varepsilon_1 < \frac{n + 2}{n - 2},
	\end{equation}
	and
	\begin{equation}
		\gamma > -\frac{n}{n - 2}.
	\end{equation}
	
	\end{itemize}
	
	\end{proof}

	\begin{proof}[Proof of lemma \ref{lemma2}]
		
		Using the Theorem 2.2 in \cite{LJY}, we only need to consider the case $\alpha \geq \frac{n}{n - 2}$ when $n \geq 4$ and $\alpha \geq \frac{n + 4}{n}$ when $n = 2, 3$. Similarly to the proof of lemma \ref{lemma1}, we let $\varepsilon_2 = 0$, then
	\begin{equation}
		\begin{aligned}
			t &= (n - 1)\gamma + (n - 2)\alpha - 1.
		\end{aligned}
	\end{equation}
	Let $\gamma = -1$, then
	\begin{equation}
		t = (n - 2)\alpha - n.
	\end{equation}
		When 
	\begin{equation}
		\frac{-2}{\sqrt n - 1} < t < \frac{2}{\sqrt n + 1},
	\end{equation}
	we have
	\begin{equation}
		\begin{aligned}
			F=&\frac{4n}{n - 1}(1 - t) - (t - 2)^2 > 0.
		\end{aligned}
	\end{equation}
	So we know when $\gamma = -1$, it follows that
	\begin{align*}
		\frac{1}{n - 2}\left(\frac{-2}{\sqrt n - 1} + n \right) < \alpha < \frac{1}{n - 2}\left(\frac{2}{\sqrt n + 1} + n\right),
	\end{align*}
	and
	\begin{equation}
		\alpha + \varepsilon_1 = \frac{n + 2}{n - 1}.
	\end{equation}
	Since $\varepsilon_1 > 0$, so we have when $n > 2$,
	\begin{equation}
		\frac{1}{n - 2}\left(\frac{-2}{\sqrt n - 1} + n \right) < \alpha < \min \left\{\frac{1}{n - 2}\left(\frac{2}{\sqrt n + 1} + n\right), \frac{n + 2}{n - 1} \right\},
	\end{equation}
	and when $n = 2$
	\begin{equation}
		0 < \alpha < \frac{n + 2}{n - 1},
	\end{equation}
	the conditions \eqref{1201equ1}, $\cdots$, \eqref{1201equ2} all hold.

	\end{proof}
	Now using these two lemmas, we can prove the Liouville theorems.
	
	\begin{proof}[Proof of Theorem \ref{theorem2}]
	 Fix $x_0 \in M$ and let $\eta$ be a smooth cut-off function such that
		\begin{equation}
			\begin{aligned}
				\eta &= 1, \text{ in } B_R(x_0),\\
				\eta &= 0, \text{ in } M\backslash B_{2R}(x_0),			
			\end{aligned}
		\end{equation}
		then
		\begin{equation}
			\begin{aligned}
				|\nabla \eta| &\leq CR^{-1},\\
				|\Delta \eta| &\leq \frac{C(1 + \rho^{2 - \epsilon})}{R^2}.
			\end{aligned}
		\end{equation}
		Define 
		\begin{equation}
			W := Z\eta^{\theta}.
		\end{equation}
		Consider in the ball $B_{2R}(x_0)$ and suppose $W$ attains its maximum at $x_1 \in B_{2R}(x_0)$. Then at $x_1$, we have
		\begin{equation*}
			W_i(x_1) = Z_i \eta^{\theta} + \theta Z \eta^{\theta - 1}\eta_i = 0,
		\end{equation*} 
		and
		\begin{align*}
			0 &\geq \frac{1}{2\beta}(U^{t - \gamma + 1}W_i)_i\\
			&= \Big( U^{t - \gamma + 1} Z_i \eta^{\theta} + \theta U^{t - \gamma + 1} Z\eta^{\theta - 1}\eta_i \Big)_i \\
		 &= \Big( U^{t - \gamma + 1} Z_i\Big)_i \eta^{\theta} + 2\theta U^{t - \gamma + 1} Z_i \eta^{\theta - 1}\eta_i  + \theta(t - \gamma + 1)  U^{t - \gamma} Z\eta^{\theta - 1}\eta_i u_i + U^{t - \gamma + 1} Z\Delta( \eta^{\theta}).
		\end{align*}
		By Lemma \ref{lemma2}, we have
		\begin{equation}
			\begin{aligned}
				(U^{t - \gamma + 1}Z_i)_i &\geq -  U^{t - \gamma + 1}Z_i b_i + \varepsilon U^{t + 2\alpha} h^2  + \varepsilon U^{t - 2}|\nabla U|^4 +  \varepsilon U^{t}E_{i j}^2.
			\end{aligned}
		\end{equation}
		So 
		\begin{align*}
			0  &\geq -  U^{t - \gamma + 1}Z_i b_i\eta^{\theta}  + \varepsilon U^{t + 2\alpha} h^2 \eta^{\theta} + \varepsilon U^{t - 2}|\nabla U|^4\eta^{\theta} +  \varepsilon U^{t}E_{i j}^2 \eta^{\theta}\\
			& + 2\theta U^{t - \gamma + 1} Z_i \eta^{\theta - 1}\eta_i + \theta(t - \gamma + 1)  U^{t - \gamma} Z\eta^{\theta - 1}\eta_i u_i + U^{t - \gamma + 1} Z\Delta( \eta^{\theta})\\
		 &\geq  \theta U^{t - \gamma + 1}Z b_i\eta^{\theta - 1}\eta_i  + \varepsilon U^{t + 2\alpha} h^2 \eta^{\theta} + \varepsilon U^{t - 2}|\nabla U|^4\eta^{\theta} +  \varepsilon U^{t}E_{i j}^2 \eta^{\theta}\\
			& - 2\theta^2 U^{t - \gamma + 1} Z \eta^{\theta - 2}|\nabla\eta|^2  + \theta(t - \gamma + 1)  U^{t - \gamma} Z\eta^{\theta - 1}\eta_i u_i + U^{t - \gamma + 1} Z\Delta( \eta^{\theta}),
		\end{align*}
		\begin{equation}\label{man_final1}
			\begin{aligned}
				\Rightarrow & \varepsilon U^{t + 2\alpha} h^2 \eta^{\theta} + \varepsilon U^{t - 2}|\nabla U|^4\eta^{\theta} +  \varepsilon U^{t}E_{i j}^2 \eta^{\theta}\\
				&\leq \underset{\textcircled{1}}{ -\theta U^{t - \gamma + 1}Z b_i\eta^{\theta - 1}\eta_i} + \underset{\textcircled{2}}{ 2\theta^2 U^{t - \gamma + 1} Z \eta^{\theta - 2}|\nabla\eta|^2} \underset{\textcircled{3}}{ - \theta(t - \gamma + 1)  U^{t - \gamma} Z\eta^{\theta - 1}\eta_i u_i} \underset{\textcircled{4}}{ - U^{t - \gamma + 1} Z\Delta( \eta^{\theta})}.
			\end{aligned}
		\end{equation}
		\begin{itemize}
			\item \textcircled{1}:
			\begin{equation}\label{term1}
				\begin{aligned}
					\textcircled{1} &= -\theta U^{t - \gamma + 1}Z b_i\eta^{\theta - 1}\eta_i\\
					&= -\theta U^{t - \gamma + 1} b_i\eta^{\theta - 1}\eta_i(hU^{\gamma + \alpha} + \beta U^{\gamma - 1}|\nabla U|^2)\\
					&\leq \varepsilon^2 U^{t + 2\alpha} h^2 \eta^{\theta} + \varepsilon^2 U^{t - 2}|\nabla U|^4\eta^{\theta}  + \frac{C}{\varepsilon^2 R^2} U^{t + 2}|b|^2\eta^{\theta - 2}.
				\end{aligned}
			\end{equation}
			\item \textcircled{2}:
			\begin{equation}\label{term2}
				\begin{aligned}
					\textcircled{2} &= 2\theta^2 U^{t - \gamma + 1} Z \eta^{\theta - 2}|\nabla\eta|^2\\
					&= 2\theta^2 U^{t - \gamma + 1}  \eta^{\theta - 2}|\nabla\eta|^2(hU^{\gamma + \alpha} + \beta U^{\gamma - 1}|\nabla U|^2)\\
					&\leq \varepsilon^2 U^{t + 2\alpha} h^2 \eta^{\theta} + \varepsilon^2 U^{t - 2}|\nabla U|^4\eta^{\theta}  + \frac{C}{\varepsilon^2 R^4} U^{t + 2}\eta^{\theta - 4}.
				\end{aligned}
			\end{equation}
			\item \textcircled{3}:
			\begin{equation}\label{term3}
				\begin{aligned}
					\textcircled{3} &=  - \theta(t - \gamma + 1)  U^{t - \gamma} Z\eta^{\theta - 1}\eta_i u_i\\
					&=  - \theta(t - \gamma + 1)  U^{t - \gamma} \eta^{\theta - 1}\eta_i u_i(hU^{\gamma + \alpha} + \beta U^{\gamma - 1}|\nabla U|^2)\\
					&\leq \varepsilon^2 U^{t + 2\alpha} h^2 \eta^{\theta} + \varepsilon^2 U^{t - 2}|\nabla U|^4\eta^{\theta}  + \frac{C}{\varepsilon^2 R^4} U^{t + 2}\eta^{\theta - 4}.
				\end{aligned}
			\end{equation}
			\item \textcircled{4}:
			\begin{equation}\label{term4}
				\begin{aligned}
					\textcircled{4} &= - U^{t - \gamma + 1} Z\Delta( \eta^{\theta})\\
					&= - U^{t - \gamma + 1} Z \Big[ \theta(\theta - 1)\eta^{\theta - 2}|\nabla \eta|^2 + \theta \eta^{\theta - 1}\Delta \eta\Big]\\
					&\leq \varepsilon^2 U^{t + 2\alpha} h^2 \eta^{\theta} + \varepsilon^2 U^{t - 2}|\nabla U|^4\eta^{\theta}  + \frac{C}{\varepsilon^2 R^4} U^{t + 2}\eta^{\theta - 4} + \frac{C}{\varepsilon^2 } U^{t + 2}\eta^{\theta - 4}|\Delta \eta|^2 .
				\end{aligned}
			\end{equation}
			
		\end{itemize}
		Therefore substituting \eqref{term1}, \eqref{term2}, \eqref{term3} and \eqref{term4} into \eqref{man_final1}, we get that at $x_1$
		\begin{equation}\label{Young1}
			\begin{aligned}
				& U^{t + 2\alpha} h^2 \eta^{\theta} +  U^{t - 2}|\nabla U|^4\eta^{\theta} \leq \frac{C}{ R^2} U^{t + 2}|b|^2\eta^{\theta - 2} + \frac{C}{ R^4} U^{t + 2}\eta^{\theta - 4} + \frac{C}{\varepsilon^2 }U^{t + 2}\eta^{\theta - 4}|\Delta \eta|^2,\\
				\Rightarrow & U^{2\gamma + 2\alpha} h^2 \eta^{2\theta} +  W^2 \leq \frac{C}{ R^2} U^{2\gamma + 2}|b|^2\eta^{2\theta - 2}  + CU^{2\gamma + 2}\eta^{2\theta - 4}|\Delta \eta|^2 ,\\
			\end{aligned}
		\end{equation}
		Since $\Delta \eta \leq \frac{C(1 + \rho^{2 - \epsilon})}{R^2}$, we obtain
		\begin{equation}\label{Young2}
			\begin{aligned}
				\Rightarrow & U^{2\gamma + 2\alpha} h^2 \eta^{2\theta} +  W^2 \leq \frac{C}{ R^2} U^{2\gamma + 2}|b|^2\eta^{2\theta - 2}  + \frac{C}{ R^{2\epsilon}}U^{2\gamma + 2}\eta^{2\theta - 4}.\\
			\end{aligned}
		\end{equation}

		Depending on the choice of $\alpha$, there are two different cases.  If   $0 < \alpha < \min \left\{\frac{1}{n - 2}\left(\frac{2}{\sqrt n + 1} + n\right), \frac{n + 2}{n - 1} \right\}$,
		we use the result of Lemma \ref{lemma2}.
		Since $\gamma = -1$, we finally get that
		\begin{equation}
			\max_{B_{R}(x_0)}Z \leq 2^{\theta}\max_{B_{R}(x_0)}W \leq 2^{\theta}\max_{B_{2R}(x_0)}W \leq 2^{\theta}W(x_1) \leq \frac{C}{R}|b|\eta^{\theta - 1} + \frac{C}{R^\epsilon}\eta^{\theta - 2}.
		\end{equation}
		Choose $\theta = 2$. Because $|b| = o(\rho)$, we get that when $R$ tends to $\infty$,
		\begin{equation}
			\max_{B_{R}(x_0)}Z \leq 0.
		\end{equation}
		which is impossible if $U > 0$.
	
		Next we consider the case that 
		\begin{equation}
			\min \left\{\frac{1}{n - 2}\left(\frac{2}{\sqrt n + 1} + n\right), \frac{n + 2}{n - 1} \right\} \leq \alpha < \frac{n + 2}{n - 2}.
		\end{equation}
		Firstly, we need to show that at this time, we can always choose $\varepsilon_1, \varepsilon_2 > 0$, such that $\gamma < -1$. We still let $\varepsilon_2 > 0$ small enough.
		
		\begin{itemize}
			\item If $\frac{1}{n - 2}\left(\frac{2}{\sqrt n + 1} + n\right) \geq \frac{n + 2}{n - 1}$, then 
			\begin{equation}
				\alpha + \varepsilon_1 > \frac{n + 2}{n - 1} \Rightarrow \gamma < -1.
			\end{equation}
			
			\item If $\frac{1}{n - 2}\left(\frac{2}{\sqrt n + 1} + n\right) < \frac{n + 2}{n - 1}$, then $n \geq 8$. Let $\varepsilon_1 = \frac{(n - 2)^2 - 4n}{(n - 2)^2(n - 1)}$. By \eqref{range} we get that when
			\begin{equation}
				\begin{aligned}
					\frac{n + 2}{n - 1} \leq \alpha + \varepsilon_1 \leq \frac{n + 2}{n - 1} \cdot \frac{(n - 2)^3 + 4n}{(n - 2)^3},
				\end{aligned}
			\end{equation}
			$F > 0$.
			By direct computation, we get that when $n \geq 8$,
			\begin{equation}
				\frac{n + 2}{n - 1} < \frac{1}{n - 2}\left(\frac{2}{\sqrt n + 1} + n\right) + \frac{(n - 2)^2 - 4n}{(n - 2)^2(n - 1)}.
			\end{equation}
			So for $\alpha > \frac{1}{n - 2}\left(\frac{2}{\sqrt n + 1} + n\right)$, we let $\varepsilon_1$ from $\frac{(n - 2)^2 - 4n}{(n - 2)^2(n - 1)}$ tend to $0$. As a result, 
			\begin{equation}
				\alpha + \varepsilon_1 > \frac{n + 2}{n - 1}, \Rightarrow \gamma < -1.
			\end{equation}
		\end{itemize}
		
	Consider in the ball $B_{2R}(x_0)$ and suppose $W$ attains its maximum at $x_1 \in B_{2R}(x_0)$. 
		If $x_1 \in B_R(x_0)$, then $Z(x_1) = W(x_1)$. So by maximum principle, we know 
		\begin{equation}
			\max_{B_{R}}Z \leq \max_{B_{2R}}W \leq W(x_1) = Z(x_1) \leq 0.
		\end{equation}
		This is impossible since $U(x_1) > 0$ and $Z(x_1) > 0$. 
		
		If $x_1 \in B_{2R}(x_0)\backslash B_R(x_0)$.
		 By \eqref{Young1}, we get that
		\begin{equation}
			\begin{aligned}
				& U^{2\gamma + 2\alpha} h^2 \eta^{2\theta} +  W^2 \leq \frac{C}{ R^2} U^{2\gamma + 2}|b|^2\eta^{2\theta - 2}  + CU^{2\gamma + 2}\eta^{2\theta - 4}|\Delta \eta|^2.
			\end{aligned}
		\end{equation}
		Since 
		\begin{equation}
			|\Delta \eta| \leq \frac{C(1 +  k\rho )}{R^2},
		\end{equation}
		we have
		\begin{equation}
			\begin{aligned}
				& W^2(x_1) \leq \frac{C}{ R^{4 }} U^{2\gamma + 2}\eta^{2\theta - 2} ,\\
				\Rightarrow & W(x_1) \leq \frac{C}{ R^2} U^{\gamma + 1}\eta^{\theta - 1}.
			\end{aligned}
		\end{equation}
		Inspired by the work of Serrin-Zou \cite{Serrin1} and Catino-Monticelli \cite{JEMS} , we need to give a lower bound of $U$. Following the proof of Lemma 2.8 in \cite{JEMS}, we define 
		\begin{equation}
			v:=  \rho^{2 - n - \delta},
		\end{equation}
		where $\delta > 0$ will be determined later.  Then
		\begin{equation}
			\begin{aligned}
				\Delta v + b_i v_i &= (n - 1 + \delta)(n - 2 + \delta) \rho^{-n - \delta} - (n - 2 + \delta) \rho^{1 - n - \delta}\Delta \rho \\
				&- (n - 2 + \delta)\rho^{1 - n - \delta} b_i \rho_i.
			\end{aligned}
		\end{equation}
		Applying the Laplacian comparison (Theorem 1.2 in \cite{Al}), we get
		\begin{equation}
			\rho \Delta \rho \leq (n - 1)(1 + c_n k \rho ).
		\end{equation}
		Thus
		\begin{equation}
			\begin{aligned}
				\Delta v + b_i v_i &\geq (n - 1 + \delta)(n - 2 + \delta) \rho^{-n - \delta} - (n - 2 + \delta) \rho^{ - n - \delta}(n - 1)(1 + c_n k \rho) \\
				&- (n - 2 + \delta)\rho^{1 - n - \delta} |b|\\
				&=  \delta(n - 2 + \delta) \rho^{-n - \delta} - (n - 2 + \delta) \rho^{ - n - \delta}(n - 1) c_n k \rho \\
				&- (n - 2 + \delta) \rho^{1 - n - \delta} |b|\\
				&= \Big[ \delta(n - 2 + \delta)  - (n - 2 + \delta) (n - 1) c_n k\rho - (n - 2 + \delta) |b|\rho \Big]\rho^{ - n - \delta }.
			\end{aligned}
		\end{equation}
		So there exists $\rho_0 = \rho_0(\delta, n, k, b)$, such that 
		when $\rho \geq \rho_0$ we have
		\begin{equation}\label{vvv}
			\Delta v + b_i v_i > 0.
		\end{equation}
		We remark that \eqref{vvv} holds pointwise in the complement of the cut locus of $x_0$ and weakly on M. So we have
		\begin{equation}
			\Delta v + b_i v_i > 0, \text{ weakly on }M\backslash B_{\rho_0}.
		\end{equation}
		 Define 
		\begin{equation}
			V := v\cdot \min_{B_{\rho_0}}U,
		\end{equation}
		then
		\begin{equation}
			U > V, \text{ in } M \backslash B_{\rho_0}.
		\end{equation}
		Now let $R > \rho_0$, then
		\begin{equation}
			\begin{aligned}
				 W(x_1) &\leq \frac{C}{ R^2} U^{\gamma + 1}\eta^{\theta - 1} \\
				  &\leq \frac{C}{ R^{2}} \eta^{\theta - 1}\rho^{(2 - n - \delta)(\gamma + 1)}\\
				  &\leq CR^{-(n - 2)\gamma - n - \delta(\gamma + 1)}.
			\end{aligned}
		\end{equation}
		Since $-\frac{n}{n - 2} < \gamma < -1 $, if we choose 
		\begin{equation}
			\delta = -\frac{(n - 2)\gamma + n}{2(\gamma + 1)},
		\end{equation}
		then
		\begin{equation}
			W(x_1)\leq CR^{-\frac{1}{2}[(n - 2)\gamma + n] }.
		\end{equation}
		Let $R$ tends to $\infty$, we get
		\begin{equation}
			\max_{B_{R}}Z \leq 0.
		\end{equation}
		This is impossible.
		
	\end{proof}

\end{CJK}	
\end{document}